\documentclass[12pt]{amsart}
\usepackage{amsfonts}
\usepackage{amsfonts,latexsym,rawfonts,amsmath,amssymb,amsthm}
\usepackage[plainpages=false]{hyperref}

\usepackage{graphicx}
\usepackage{amsmath}
\allowdisplaybreaks[4]
\RequirePackage{color}

\numberwithin{equation}{section}

\newcommand{\beq}{\begin{equation}}
\newcommand{\eeq}{\end{equation}}
\newcommand{\beqs}{\begin{eqnarray*}}
\newcommand{\eeqs}{\end{eqnarray*}}
\newcommand{\beqn}{\begin{eqnarray}}
\newcommand{\eeqn}{\end{eqnarray}}
\newcommand{\beqa}{\begin{array}}
\newcommand{\eeqa}{\end{array}}

\newtheorem{Proposition}{Proposition}[section]
\newtheorem{Theorem}[Proposition]{Theorem}
\newtheorem{Lemma}[Proposition]{Lemma}

\newtheorem{Corollary}[Proposition]{Corollary}

\newtheorem{definition}{Definition}

\newtheorem{remark}{Remark}

\title  {Devaney chaos in a Two-Dimensional Space with Weak Topology}

\begin{document}



\email{zenghongbo@csust.edu.cn. }


\bibliographystyle{plain}

\maketitle

\baselineskip=15.8pt
\parskip=3pt

\centerline {\bf   Hongbo Zeng}
\centerline {School of Mathematics and Statistics, Changsha University of Science and Technology, China}

\vskip20pt

\noindent {\bf Abstract}:

It is well known that a finite-dimensional linear system cannot be chaotic. In this article, it shows that Devaney chaos with the weak topology can be generated by a linear map, where the weak topology into a two-dimensional Euclidean space is induced by a linear functional. Especially, it gives the equivalent conditions to be  strongly transitive (weakly transitive, strongly sensitive,weakly sensitive, dense periodic points, respectively). Besides, we show that for a matrix with two conjugate complex eigenvalues, strong transitivity (sensitivity, respectively) is equivent to weak transitivity (sensitivity, respectively), which does not hold if we consider a matrix with two real eigenvalues. Finally,  in two-dimensional space with weak topology,  we show that weak transitivity and weak density of periodic points do not imply weak sensitivity and that weak transitivity and fixed point do not imply Li-Yorke chao, which show that there is a significant difference between the dynamic properties of weak topology and the dynamic properties of norm topology.


 \vskip20pt
 \noindent{\bf Key Words:}Finite-dimensional space; Devaney chaos; linear system;  weak topology.
 \vskip20pt

\vskip20pt

\baselineskip=15.8pt
\parskip=3pt

\newpage

\maketitle

\baselineskip=15.8pt
\parskip=3.0pt

\section{Introduction}

The discovery of chaos phenomenon and the initiation of chaos theory are one of the
greatest scientific discoveries in the last century. The study on chaos has become a major
project in nonlinear science, which has gotten a rapid development and rich achievements. There exists a chaotic phenomenon in almost all fields relating to dynamical progress and the chaos theory is a hot topic
in area of topological dynamics, which has had a great influence on modern
science including natural science and many humanities. Li and Yorke gave the definition of chaos first in 1975 (see \cite{a1}). Since then, different people from different
fields gave different definitions of chaos under their
understanding of the subject, such as  distributional chaos (i.e. Schweizer-Smital chaos, see \cite{a2}), Devaney chaos \cite{a3}, etc. In 1986, Devaney
proposed the widely accepted definition of chaos (topological transitivity, dense periodic
points and sensitivity), and emphasized the significance of sensitivity in describing dynamical systems. Moreover, in 1992, Banks et al found that in the three conditions defining
Devaney chaos, topological transitivity and dense periodic points together
imply sensitivity if the dynamical systems has no isolated points \cite{q09}. Sensitive dependence on initial conditions (briefly, sensitivity), first defined in \cite{q9}, is one of the most remarkable components of dynamical systems theory, which characterizes the unpredictability in chaotic phenomena and is an integral part of different types of chaos. Besides, as we all know, topological transitivity (shortly,
transitivity) has been an eternal topic in the study of topological dynamical systems,
which is another crucial measure of system complexity. We can see more instances in \cite{q91,q92}.

Chaos has received considerable attention in recent years. Chaos occurs in both finite- and infinite-
dimensional nonlinear systems. Please see \cite{q93,q94} for the existence of chaos for the finite-dimensional
nonlinear systems governed by ordinary differential equations, and \cite{q95,q96,q97} for chaotic results
on the infinite-dimensional nonlinear systems governed by partial differential equations.

Linear systems can also exhibit chaos, but they must be considered in the infinite-dimensional spaces. In
1991, Godefroy and Shapiro  \cite{q98} first introduced the Devaney chaos in linear systems. Since then a large
number of investigators have been dedicated in Devaney chaotic operators and $C_0$-semigroups. Regarding the chaos of linear systems, there are two other generally accepted concepts of
chaos: Li-Yorke chaos \cite{q99} and distributional chaos \cite{q89}. These two notions of chaos in the context of linear
systems have been studied. On the one hand, linear chaos is linked to infinite-
dimensional spaces. The reason is that a continuous map on a compact metric space is conjugate to the restriction of a linear operator on certain invariant set. More specifically, the same operator can be taken
for all nonlinear systems, and the operator is even chaotic \cite{q88}. On the other hand, the finite-dimensional
linear systems are not chaotic in the setting with the norm topology, see for example  \cite{q89,q88}.

When Banach spaces are endowed with the weak topology, some interesting and important results on
functional analysis are established \cite{q87}. Moreover, in the theory of Hilbert spaces, weak topology plays an
essential role \cite{q86}. However, to our best knowledge, too little attention has been paid to the complexity of
both finite- and infinite-dimensional linear systems under the weak topology. For the infinite-dimensional
linear systems, Shkarin \cite{q85} showed that there exists a weakly supercyclic unitary operator on a separable
Hilbert space, which is not weakly sequentially supercyclic. Bayart and Matheron \cite{q84} established a Weak
Angle Criterion for the existence of supercyclic vector in weak topology on a separable Banach space.
For the finite-dimensional linear systems, Zhang and Chen \cite{q83} proved the existence of Li-Yorke chaos
in weak topology for linear maps on $\mathbb{R}^2$, while the chaotic behavior in weak topology is unclear for the one-dimensional case. Meanwhile, they obtained some sufficient conditions for the Li-Yorke chaos in weak topology.  Subsequent work by Zhu and Yang extended this investigation to $n$-dimensional Euclidean spaces endowed with the weak topology \cite{q82}.

The purpose of this work is to study the strong and weak Devaney chaos in weak topology for the two-dimensional linear
systems. More precisely, by introducing a weak topology defined by a family of semi-norms, the complicated
dynamics of the linear systems on $\mathbb{R}^2$ with such weak topology is systematically studied. Some equivalent conditions for strong (weak) Devaney chaos in weak topology are established. These conditions further
clarify the existence and mechanism of chaos for the linear systems on $\mathbb{R}^2$ with the weak topology. Compared with  Li-Yorke chaos in weak topology for linear maps on $\mathbb{R}^2$ \cite{q83}, in this paper we have completely solved the problem of determining Devaney chaos in weak topology for linear maps in a two-dimensional space.

The rest of this paper is organized as follows. In Section 2,  we will state some preliminaries,
definitions and some lemmas. The main conclusions will be given in the last four sections.   Section 3 focuses on the investigation of transitivity in weakly topology in two-dimensional Euclidean spaces. The  periodic points are considered in Section 4.  In Section 5,  we study the sensitivity in weakly topology in two-dimensional Euclidean spaces. And in finally section, we present the Devaney chaos in weakly topology in two-dimensional Euclidean spaces.

\section{Preparations and basic Lemmas}

For the sake of convenience, in this section, some useful concepts and results are summarized (one can see \cite{q83}).

\subsection{Linear algebra}
In this subsection, we go over certain basic concepts and lemmas from linear algebra, which are sourced from \cite{q81}.

 Recall that for a $2\times2$ real matrix $\begin{pmatrix}a_{11}&a_{12}\\a_{21}&a_{22}\end{pmatrix}$, its eigenvalues are the solutions to the subsequent equation:
	\[
	\begin{vmatrix}
		a_{11}-\lambda&a_{12}\\
		a_{21}&a_{22}-\lambda
	\end{vmatrix}
	=\lambda^{2}-(a_{11} + a_{22})\lambda+(a_{11}a_{22}-a_{12}a_{21}).
	\]

\begin{definition}
	Two matrices $A$ and $B$ are similar if there is a nonsingular matrix $C$ such that $A=CBC^{-1}$.
	\end{definition}
\begin{Lemma}[ Jordan Canonical Form]
 Let $A$ be a $2\times2$ matrix. Then, there exists an invertible matrix
 	$S$ such that $A = SBS^{-1}$, where $B$ has one of the following forms:
 	\[
 	\begin{pmatrix}
 		\lambda_1 & 0 \\
 		0 & \lambda_2
 	\end{pmatrix}
 	\quad\text{and}\quad
 	\begin{pmatrix}
 		\lambda & 1 \\
 		0 & \lambda
 	\end{pmatrix}.
 	\]
\end{Lemma}

\begin{Lemma}\cite{q93}
		Let $A$ be a $2\times2$ real matrix with real eigenvalues. Then, there exists a real invertible matrix $S$ such that $A = SBS^{-1}$, where $B$ has one of the following forms:
		\[
		\begin{pmatrix}
			\lambda_1 & 0 \\
			0 & \lambda_2
		\end{pmatrix}
		\quad\text{and}\quad
		\begin{pmatrix}
			\lambda & 1 \\
			0 & \lambda
		\end{pmatrix},
		\]
		where $\lambda,\lambda_1$, and $\lambda_2$ are real numbers.
	\end{Lemma}

\begin{Lemma}\cite{q93}
		Let $A$ be a $2\times2$ real matrix with conjugate complex eigenvalues. Then, there exists a real invertible matrix $S$ such that $A = SBS^{-1}$, where 
		\[
		B=\begin{pmatrix}
			a & b \\
			-b & a
		\end{pmatrix}
		.
		\]
		Besides, $B$ can be expressed as
\[\begin{pmatrix}
		cos(2\pi\theta) & -sin(2\pi\theta) \\
		sin(2\pi\theta) & cos(2\pi\theta)
	\end{pmatrix}
	\begin{pmatrix}
		r & 0 \\
		0 & r
	\end{pmatrix}
	\]
where $r=\sqrt{a^2+b^2}$ and $\theta\in[0,1]$ is the counterclockwis angle from the positive $x$-axis to the vector $\left({\begin{array}{cc}
a\\
b\\
\end{array}}\right)$. 

	\end{Lemma}

\subsection{Functional analysis}
In the present subsection, a review is conducted on certain fundamental concepts originating from functional analysis, as documented in  \cite{q80}.
	
Let $\mathbb{F}$ denote either the real - field $\mathbb{R}$ or the complex - field $\mathbb{C}$. For simplicity, consider only the case $\mathbb{F}=\mathbb{R}$ in this paper.
	
	\begin{definition}
		Let $\mathcal{X}$ be a vector space over $\mathbb{F}=\mathbb{R}$. An inner product on $\mathcal{X}$ is a function $\langle\cdot,\cdot\rangle:\mathcal{X}\times\mathcal{X}\to\mathbb{F}$ such that, for any $\alpha,\beta\in\mathbb{F}$ and $x,y,z$ in $\mathcal{X}$, the following hold:
		\begin{enumerate}
			\item $\langle\alpha x+\beta y,z\rangle=\alpha\langle x,z\rangle+\beta\langle y,z\rangle$,
			\item $\langle x,\alpha y+\beta z\rangle=\alpha\langle x,y\rangle+\beta\langle x,z\rangle$,
			\item $\langle x,x\rangle\geq0$,
			\item $\langle x,y\rangle=\langle y,x\rangle$,
			\item $\langle x,x\rangle = 0$ implies $x = 0$.
		\end{enumerate}
	\end{definition}
	
	\begin{definition}
		A Hilbert space is a vector space $\mathcal{H}$ over $\mathbb{F}$ together with an inner product $\langle\cdot,\cdot\rangle$ such that relative to the metric $d(x,y)=\langle x - y,x - y\rangle^{1/2}=\|x - y\|$ induced by the norm, $\mathcal{H}$ is a complete metric space.
	\end{definition}
	
	\begin{Lemma}
		(Riesz Representation Theorem) Let $\mathcal{H}$ be a Hilbert space. For a bounded linear functional $L:\mathcal{H}\to\mathbb{F}$, there is a unique vector $h_0$ in $\mathcal{H}$ such that $L(h)=\langle h,h_0\rangle$ for all $h\in\mathcal{H}$.
	\end{Lemma}
\begin{definition}
		A topological vector space is a vector space $\mathcal{X}$ together with a topology such that, with respect to this topology:
		\begin{enumerate}
			\item The map of $\mathcal{X}\times\mathcal{X}\to\mathcal{X}$ defined by $(x,y)\to x + y$ is continuous;
			\item The map of $\mathbb{F}\times\mathcal{X}\to\mathcal{X}$ defined by $(\alpha,x)\to\alpha x$ is continuous,
		\end{enumerate}
		where $\overline{\alpha}=\alpha$ since $\mathbb{F}=\mathbb{R}$ in this paper.
\end{definition}
\begin{definition}
		A normed space is defined as a pair $(\mathcal{X},\|\cdot\|)$, where $\mathcal{X}$ represents a vector space and $\|\cdot\|$ is a norm defined on $\mathcal{X}$. A Banach space is a normed space that is complete with respect to the metric induced by the norm.
	\end{definition}
	
	\begin{definition}
		A locally convex space is a topological vector space. Its topology is determined by a family of seminorms $\mathcal{P}$ such that $\bigcap_{p\in\mathcal{P}}\{x : p(x)=0\}=\{0\}$.
	\end{definition}
	
	\begin{remark}
		Let $\mathcal{X}$ be a normed space and $\mathcal{X}^*$ denote the set of all bounded linear functionals on $\mathcal{X}$. For every $x^*\in\mathcal{X}^*$, define $p_{x^*}(x)=|x^*(x)|$. Then, $p_{x^*}$ is a seminorm. Moreover, if $\mathcal{P}=\{p_{x^*}:x^*\in\mathcal{X}^*\}$, then $\mathcal{P}$ endows $\mathcal{X}$ with the structure of a locally convex space. The topology on $\mathcal{X}$ defined by these seminorms is known as the weak topology, and this paper only considers this weak topology.
		
		For a locally convex space $\mathcal{X}$, let $\mathcal{X}^*$ be the space of continuous (bounded) linear functionals on $\mathcal{X}$. For any $x^*,y^*\in\mathcal{X}^*,x\in\mathcal{X}$ and $\alpha\in\mathbb{F}$, define $(\alpha x^* + y^*)(x)=\alpha x^*(x)+y^*(x)$. This definition induces a natural vector - space structure on $\mathcal{X}^*$. For convenience, we use $\langle x,x^*\rangle$ to denote $x^*(x)$ for $x\in\mathcal{X}$ and $x^*\in\mathcal{X}^*$.
	\end{remark}

\begin{definition}
		Let $\mathcal{X}$ be a locally convex space and $\mathcal{X}^*$ be the space of continuous linear functionals on $\mathcal{X}$. Then, the weak topology on $\mathcal{X}$ is specified by the collection of seminorms $\{p_{x^*}:x^*\in\mathcal{X}^*\}$, where $p_{x^*}(x)=\vert\langle x,x^*\rangle\vert$ for $x\in\mathcal{X}$ and $x^*\in\mathcal{X}^*$.
	\end{definition}

\subsection{Transitivity, sensitivity and Devaney chaos}
In the final subsection, the notions of strong and weak transitivity (strong and weak sensitivity, strongly and weakly dense periodic point,  strong and weak Devaney chaos,  respectively) in the weak topology is introduced. Denote $\mathbb{N}=\{1,2,3,...\}$.

\begin{definition}
	Consider a locally convex space  $\mathcal{X}$ and an operator $T: \mathcal{X}\to\mathcal{X}$. $T$ is called strongly sensitive in weakly topology if there is $\delta>0$  such that for any nonzero $x^*\in\mathcal{X}^*$, any $x_0\in \mathcal{X}$ and any $\varepsilon>0$, there exist $x,y\in\mathcal{X}$ and $n\in\mathbb{N}$ with $p_{x^*}(x-x_0)<\varepsilon$ and $p_{x^*}(y-x_0)<\varepsilon$ satisfying $p_{x^{*}}(T^{n}x-T^{n}y)>\delta$.  $T$ is called weakly sensitive in weakly topology if there is $\delta>0$ and a nonzero $x^*\in\mathcal{X}^*$ such that for any $x_0\in \mathcal{X}$ and any $\varepsilon>0$, there exist $x,y\in\mathcal{X}$ and $n\in\mathbb{N}$ with $p_{x^*}(x-x_0)<\varepsilon$ and $p_{x^*}(y-x_0)<\varepsilon$ satisfying $p_{x^{*}}(T^{n}x-T^{n}y)>\delta$.
	\end{definition}

\begin{definition}
Consider a locally convex space  $\mathcal{X}$ and an operator $T: \mathcal{X}\to\mathcal{X}$. $T$ is said to be weakly transitive in weak topology, if there exists a nonzero $x^*\in\mathcal{X}^*$ such that for any $x,y\in \mathcal{X}$ and $\varepsilon>0$, there exists $n\in \mathbb{N}$ such that $T^n(p_{x^*}(x,\varepsilon_1))\cap p_{x^*}(y,\varepsilon_2)\neq\emptyset$, where $p_{x^*}(x,\varepsilon)=\{z\in \mathcal{X}\mid p_{x^*}(x-z)<\varepsilon\}$. Equivalently, if for any $x,y\in \mathcal{X}$ and $\varepsilon>0$, there exist $n\in \mathbb{N}$ and $z\in \mathcal{X}$ such that $|\langle z-x,x^*\rangle|<\varepsilon$ and $|\langle T^n(z)-y,x^*\rangle|<\varepsilon$. $T$  is said to be strongly transitive in weak topology, if for any nonzero $x^*\in\mathcal{X}^*$, for any $x,y\in \mathcal{X}$ and for any $\varepsilon>0$, there exists $n\in \mathbb{N}$ such that $T^n(p_{x^*}(x,\varepsilon))\cap p_{x^*}(y,\varepsilon)\neq\emptyset$.

\end{definition}

\begin{definition}
Consider a locally convex space  $\mathcal{X}$ and an operator $T: \mathcal{X}\to\mathcal{X}$. A point $x\in \mathcal{X}$ is called $k$-periodic if there exists a $k\in \mathbb{N}$ such that $T^{k}(x)=x$. If the period is 1, it is a fixed point. The operator $T$  is said to has weakly dense periodic points in weak topology, if there exists a nonzero $x^*\in\mathcal{X}^*$ such that for any $x\in \mathcal{X}$ and $\varepsilon>0$, there exists a periodic point $y\in \mathcal{X}$ such that $p_{x^*}(x-y)<\varepsilon$. The operator $T$  is said to be has strongly dense periodic points in weak topology, if for any nonzero $x^*\in\mathcal{X}^*$, for any $x\in \mathcal{X}$ and for any $\varepsilon>0$, there exists a periodic point $y\in \mathcal{X}$ such that $p_{x^*}(x-y)<\varepsilon$.

\end{definition}

\begin{definition}
Consider a locally convex space  $\mathcal{X}$ and an operator $T: \mathcal{X}\to\mathcal{X}$. $T$  is said to be weakly Devaney chaos in weak topology if it is weakly transitive in weak topology, has weakly dense periodic points in weak topology, and is  weakly sensitive in weak topology.  $T$  is said to be strongly Devaney chaos in weak topology if it is strongly transitive in weak topology, has strongly dense periodic points in weak topology, and is strongly sensitive in weak topology.
\end{definition}

\begin{definition}(\cite{q83})
		Consider a locally convex space  $\mathcal{X}$ and an operator $T: \mathcal{X}\to\mathcal{X}$. $T$ is said to be weakly Li-Yorke chaos in weak topology if there are nonzero $x^*\in\mathcal{X}^*$ and a uncountable scrambled set $S$ of $\mathcal{X}$ such that for any $x,y\in S$ satisfying
		\[
		\liminf_{k \to \infty} d_{x^*}(T^k x - T^k y) = 0 \quad \text{and} \quad \limsup_{k \to \infty} d_{x^*}(T^k x - T^k y) > 0.
		\]
$T$ is said to be strongly Li-Yorke chaos in weak topology if  there is a uncountable scrambled set $S$ of $\mathcal{X}$ such that for any $x,y\in S$ and for any nonzero $x^*\in\mathcal{X}^*$ satisfying
		\[
		\liminf_{k \to \infty} d_{x^*}(T^k x - T^k y) = 0 \quad \text{and} \quad \limsup_{k \to \infty} d_{x^*}(T^k x - T^k y) > 0.
		\]

	\end{definition}


\section{Transitive in weakly topology}
  This section focuses on the investigation of transitivity in weakly topology in two-dimensional Euclidean spaces.

\begin{Theorem}\label{th2}
Suppose that
$$
A=\left({\begin{array}{cc}
\lambda_1&0\\
0&\lambda_2\\
\end{array}}\right),
$$
where $\lambda_1$ and $\lambda_2$ are real numbers. Consider the dynamical system $Tx=Ax$ with $x\in R^2$.

(i) If $\lambda_1\neq \lambda_2$, then for any $x^*=\left({\begin{array}{cc}
x_1^*\\
x_2^*\\
\end{array}}\right)$ with $x_1^*x_2^*\neq0$, the system $T$ is weakly transitive for $x^*$ in weak topology.

(ii) If $\lambda_1=\lambda_2$, then the system $T$ is not weakly transitive for $x^*$ in weak topology.

(iii) For any $\lambda_1,\lambda_2$, the system $T$ is not strongly transitive  in weak topology.
\end{Theorem}

\begin{proof}
(i) For the given $x^*=\left({\begin{array}{cc}
x_1^*\\
x_2^*\\
\end{array}}\right)$ with $x_1^*x_2^*\neq0$, and for any $x_0=\left({\begin{array}{cc}
x_1^0\\
x_2^0\\
\end{array}}\right),y_0=\left({\begin{array}{cc}
y_1^0\\
y_2^0\\
\end{array}}\right)\in \mathbb{R}^2$ and $\varepsilon>0$, take
$$x=\left({\begin{array}{cc}
\frac{x_1^*(\lambda_2x_1^0-y_1^0)+x_2^*(\lambda_2x_2^0-y_2^0)}{x_1^*(\lambda_2-\lambda_1)}\\
\frac{x_1^*(y_1^0-\lambda_1x_1^0)+x_2^*(y_2^0-\lambda_1x_2^0)}{x_2^*(\lambda_2-\lambda_1)}\\
\end{array}}\right)\in \mathbb{R}^2.$$
Then it is easy to check that $|\langle x-x_0,x^*\rangle|=0<\varepsilon$ and $|\langle T(x)-y_0,x^*\rangle|=0<\varepsilon$.
Therefore, $T$ is weakly transitive for $x^*$ in weak topology.

(ii) If $\lambda_1=\lambda_2>1$ or $\lambda_1=\lambda_2<-1$, then for any nonzero $x^*\in\mathbb{R}^{2*}$, take
$$x_0=\left({\begin{array}{cc}
\frac{4x_1^*}{|x^*|}\\
\frac{4x_2^*}{|x^*|}\\
\end{array}}\right)$$
and
$$y_0=\left({\begin{array}{cc}
\frac{2x_1^*}{|x^*|}\\
\frac{2x_2^*}{|x^*|}\\
\end{array}}\right)\in \mathbb{R}^2.$$
Take $\varepsilon=|x^*|$. It is easy to see that $T^n(p_{x^*}(x_0,\varepsilon))\cap p_{x^*}(y_0,\varepsilon)=\emptyset$ for any $n\in \mathbb{N}$. Therefore, $T$ is not weakly transitive for $x^*$ in weak topology.

If $-1\le\lambda_1=\lambda_2\le 1$, then for any nonzero $x^*\in\mathbb{R}^{2*}$, take
$$y_0=\left({\begin{array}{cc}
\frac{4x_1^*}{|x^*|}\\
\frac{4x_2^*}{|x^*|}\\
\end{array}}\right)$$
and
$$x_0=\left({\begin{array}{cc}
\frac{2x_1^*}{|x^*|}\\
\frac{2x_2^*}{|x^*|}\\
\end{array}}\right)\in \mathbb{R}^2.$$
Take $\varepsilon=|x^*|$. It is easy to see that $T^n(p_{x^*}(x_0,\varepsilon))\cap p_{x^*}(y_0,\varepsilon)=\emptyset$ for any $n\in \mathbb{N}$. Therefore, $T$ is not weakly transitive for $x^*$ in weak topology.


(iii) Take $x^*=\left({\begin{array}{cc}
1\\
0\\
\end{array}}\right)$. If $\lambda_1>1$, then take $x_0=\left({\begin{array}{cc}
4\\
0\\
\end{array}}\right),y_0=\left({\begin{array}{cc}
2\\
0\\
\end{array}}\right)\in \mathbb{R}^2$ and $\varepsilon=1$. It is easy to see that $T^n(P_{x^*}(x_0,\varepsilon))\cap P_{x^*}(y_0,\varepsilon)=\emptyset$ for any $n\in \mathbb{N}$. Therefore, $T$ is not strongly transitive for $x^*$ in weak topology.

If $0\leq\lambda_1\leq1$, then take $x_0=\left({\begin{array}{cc}
2\\
0\\
\end{array}}\right),y_0=\left({\begin{array}{cc}
4\\
0\\
\end{array}}\right)\in  \mathbb{R}^2$ and $\varepsilon=1$. It is easy to see that $T^n(p_{x^*}(x_0,\varepsilon))\cap p_{x^*}(y_0,\varepsilon)=\emptyset$ for any $n\in \mathbb{N}$. Therefore, $T$ is not strongly transitive for $x^*$ in weak topology.

If $\lambda_1<0$, with the similar argument, we can get that $T$ is not strongly transitive for $x^*$ in weak topology.

\end{proof}

\begin{Theorem}\label{th3}
Suppose that
$$
A=\left({\begin{array}{cc}
\lambda&1\\
0&\lambda\\
\end{array}}\right),
$$
where $\lambda$ is a real number. Consider the dynamical system $Tx=Ax$ with $x\in  \mathbb{R}^2$, then

(i) for any $\lambda$ and for any $x^*=\left({\begin{array}{cc}
x_1^*\\
x_2^*\\
\end{array}}\right)$ with $x_1^*x_2^*\neq0$, the system $T$ is weakly transitive for $x^*$ in weak topology.

(ii) for any $\lambda$, the system $T$ is not strongly transitive for $x^*$ in weak topology.
\end{Theorem}

\begin{proof}
(i) For the given $x^*=\left({\begin{array}{cc}
x_1^*\\
x_2^*\\
\end{array}}\right)$ with $x_1^*x_2^*\neq0$, and for any $x_0=\left({\begin{array}{cc}
x_1^0\\
x_2^0\\
\end{array}}\right),y_0=\left({\begin{array}{cc}
y_1^0\\
y_2^0\\
\end{array}}\right)\in  \mathbb{R}^2$ and $\varepsilon>0$, take
$$x=\left({\begin{array}{cc}
\frac{x_1^*(y_1^0-\lambda x_1^0)+x_2^*(y_2^0-\lambda x_2^0)}{x_1^*}\\
\frac{x_1^*((1+\lambda)x_1^0-y_1^0)+x_2^*((1+\lambda)x_2^0-y_2^0)}{x_2^*}\\
\end{array}}\right)\in \mathbb{R}^2.$$
Then it is easy to check that $|\langle x-x_0,x^*\rangle|=0<\varepsilon$ and $|\langle T(x)-y_0,x^*\rangle|=0<\varepsilon$.
Therefore, $T$ is weakly transitive for $x^*$ in weak topology.

(ii) Take $x^*=\left({\begin{array}{cc}
1\\
0\\
\end{array}}\right)$. If $\lambda_1>1$, then take $x_0=\left({\begin{array}{cc}
4\\
0\\
\end{array}}\right),y_0=\left({\begin{array}{cc}
2\\
0\\
\end{array}}\right)\in \mathbb{R}^2$ and $\varepsilon=1$. It is easy to see that $T^n(P_{x^*}(x,\varepsilon))\cap P_{x^*}(y,\varepsilon)=\emptyset$ for any $n\in \mathbb{N}$. Therefore, $T$ is not strongly transitive for $x^*$ in weak topology.

If $0\leq\lambda_1\leq1$, then take $x_0=\left({\begin{array}{cc}
2\\
0\\
\end{array}}\right),y_0=\left({\begin{array}{cc}
4\\
0\\
\end{array}}\right)\in  \mathbb{R}^2$ and $\varepsilon=1$. It is easy to see that $T^n(p_{x^*}(x_0,\varepsilon))\cap p_{x^*}(y_0,\varepsilon)=\emptyset$ for any $n\in \mathbb{N}$. Therefore, $T$ is not strongly transitive for $x^*$ in weak topology.

If $\lambda_1<0$, with the similar argument, we can get that $T$ is not strongly transitive for $x^*$ in weak topology.

\end{proof}

\begin{Theorem}\label{th4}
Suppose that
$$
A=\left({\begin{array}{cc}
cos2\pi\theta&-sin2\pi\theta\\
sin2\pi\theta&cos2\pi\theta\\
\end{array}}\right)\left({\begin{array}{cc}
r&0\\
0&r\\
\end{array}}\right),
$$
where $r>0$  and $\theta\in[0,1]$. Consider the dynamical system $Tx=Ax$ with $x\in  \mathbb{R}^2$.

(i) If $\theta\notin\{0,\frac{1}{2},1\}$, then the system $T$ is weakly transitive for $x^*$ in weak topology.

(ii) If $\theta\in\{0,\frac{1}{2},1\}$, then the system $T$ is not weakly transitive for $x^*$ in weak topology.

(iii) If $\theta\notin\{0,\frac{1}{2},1\}$, then the system $T$ is strongly transitive for $x^*$ in weak topology.
\end{Theorem}

\begin{proof}
(i) Take $x^*=\left({\begin{array}{cc}
1\\
0\\
\end{array}}\right)$. Note that for any $x=\left({\begin{array}{cc}
x_1\\
x_2\\
\end{array}}\right)\in  \mathbb{R}^2$,  $p_{x^*}(x,\varepsilon_1)=p_{x^*}(x',\varepsilon)$, where  $x'=\left({\begin{array}{cc}
x_1\\
0\\
\end{array}}\right)\in  \mathbb{R}^2$. Therefore, without loss of generality, assume 
$x_0=\left({\begin{array}{cc}
u\\
0\\
\end{array}}\right)\in  \mathbb{R}^2$ and
$y_0=\left({\begin{array}{cc}
v\\
0\\
\end{array}}\right)\in  \mathbb{R}^2$, where $u<v$ and $u,v\in  \mathbb{R}$.

Let $\varepsilon>0$. If $r>1$, then there exists $n\in \mathbb{N}$ large enough such that $r^ncos2\pi n\theta>5$ and $2\pi n\theta\in(0,\frac{\pi}{2})$. Hence, there exists $w=\left({\begin{array}{cc}
r^n\\
k\\
\end{array}}\right)\in  \mathbb{R}^2$ such that

$$
\left({\begin{array}{cc}
cos2\pi n\theta&-sin2\pi n\theta\\
sin2\pi n\theta&cos2\pi n\theta\\
\end{array}}\right)\left({\begin{array}{cc}
r^n\\
k\\
\end{array}}\right)=\left({\begin{array}{cc}
5\\
l\\
\end{array}}\right),
$$
where $k,l\in \mathbb{R}$. Take $z=\left({\begin{array}{cc}
1\\
j\\
\end{array}}\right)\in  \mathbb{R}^2$, where $j\in  \mathbb{R}$ satisfies $\frac{j}{k}=\frac{1}{r^n}$. Then
it is easy to check that $|\langle z-x_0,x^*\rangle|=0<\varepsilon$ and $|\langle T^n(z)-y_0,x^*\rangle|=0<\varepsilon$.

If $r\leq1$, then there exists $n\in \mathbb{N}$ such that $2\pi n\theta\in(0,\frac{\pi}{2})$. Hence, there exists $k'\in  \mathbb{R}$ large enough such that
$$
\left({\begin{array}{cc}
cos2\pi n\theta&-sin2\pi n\theta\\
sin2\pi n\theta&cos2\pi n\theta\\
\end{array}}\right)\left({\begin{array}{cc}
r^n\\
k'r^n\\
\end{array}}\right)=\left({\begin{array}{cc}
5\\
l'\\
\end{array}}\right),
$$
where $l'\in  \mathbb{R}$. Take $z=\left({\begin{array}{cc}
1\\
k'\\
\end{array}}\right)\in  \mathbb{R}^2$. Then
it is easy to check that $|\langle z-x_0,x^*\rangle|=0<\varepsilon$ and $|\langle T^n(z)-y_0,x^*\rangle|=0<\varepsilon$.

For
$x_0=\left({\begin{array}{cc}
u\\
0\\
\end{array}}\right)\in  \mathbb{R}^2$ and
$y_0=\left({\begin{array}{cc}
v\\
0\\
\end{array}}\right)\in  \mathbb{R}^2$, where $u\geq v$ and $u,v\in R$, using the similar argument, we can get that the result still holds.

Therefore, $T$ is weakly transitive for $x^*$ in weak topology.

(ii) From Theorem \ref{th2}(ii), we can get the result.

(iii) Following the proof of $(i)$, we can obtain the result.

\end{proof}

By the Theorem \ref{th4}, we have the following  corollary.
\begin{Corollary}
Consider the dynamical system $Tx=Ax$, where $x\in\mathbb{R}^2$ and $A$ is a $2\times 2$ real matrix with two conjugate complex eigenvalues,
	\[\begin{pmatrix}
		cos(2\pi\theta) & -sin(2\pi\theta) \\
		sin(2\pi\theta) & cos(2\pi\theta)
	\end{pmatrix}
	\begin{pmatrix}
		r & 0 \\
		0 & r
	\end{pmatrix}
	\]
where $r>0$ and $\theta\in[0,1]$.

Then $T$ is strongly transitive in weak topology if and only if $T$ is weakly transitive in weak topology.

\end{Corollary}

\section{periodic points in weakly topology}
  This section focuses on the investigation of dense periodic points in weakly topology in two-dimensional Euclidean spaces.

\begin{Theorem}\label{th5}
Suppose that
$$
A=\left({\begin{array}{cc}
\lambda_1&0\\
0&\lambda_2\\
\end{array}}\right),
$$
where $\lambda_1$ and $\lambda_2$ are real numbers. Consider the dynamical system $Tx=Ax$ with $x\in  \mathbb{R}^2$.

(i) If $|\lambda_1|=|\lambda_2|=1$, then all points of $ \mathbb{R}^2$ are periodic points. Further, the system $T$ has strongly dense periodic points.

(ii) If $|\lambda_1|\neq 1$ and $|\lambda_2|=1$, then the each point $x=\left({\begin{array}{cc}
0\\
x_2\\
\end{array}}\right)\in  \mathbb{R}^2$ is periodic point, where $x_2\in  \mathbb{R}$, and the other points are not periodic points. Further, the system $T$ has strongly dense periodic points.

(iii) If $|\lambda_1|=1$ and $|\lambda_2|\neq 1$,then the every point $x=\left({\begin{array}{cc}
x_1\\
0\\
\end{array}}\right)\in  \mathbb{R}^2$ is periodic point, where $x_1\in  \mathbb{R}$, and the other points are not periodic points. Further, the system $T$ has strongly dense periodic points.

(iv)If $|\lambda_1|\neq1$ and $|\lambda_2|\neq 1$, then all points $x=\left({\begin{array}{cc}
x_1\\
x_2\\
\end{array}}\right)\in  \mathbb{R}^2$ with $x_1^2+x_2^2\neq0$ are not periodic points. Further, the system $T$ does not have weakly dense periodic points.

\end{Theorem}

\begin{proof}
(i) Assume $\lambda_1=\lambda_2=1$, then the matrix
$$
A=\left({\begin{array}{cc}
1&0\\
0&1\\
\end{array}}\right).
$$
For any $x\in  \mathbb{R}^2$, $T(x)=A(x)=\left({\begin{array}{cc}
1&0\\
0&1\\
\end{array}}\right)x=x$. Therefore, all points of $ \mathbb{R}^2$ are periodic points.

When $\lambda_1=-\lambda_2=1$, then the matrix
$$
A=\left({\begin{array}{cc}
1&0\\
0&-1\\
\end{array}}\right).
$$
For any $x\in  \mathbb{R}^2$, $T^2(x)=A^2(x)=\left({\begin{array}{cc}
1&0\\
0&1\\
\end{array}}\right)x=x$. Therefore, all points of $ \mathbb{R}^2$ are periodic points.  Further, the system $T$ has strongly dense periodic points.

When $\lambda_1=-\lambda_2=-1$ or  $\lambda_1=\lambda_2=-1$, with the similar argument, 
for any $x\in  \mathbb{R}^2$, $T^2(x)=A^2(x)=\left({\begin{array}{cc}
1&0\\
0&1\\
\end{array}}\right)x=x$. Therefore, all points of $ \mathbb{R}^2$ are periodic points.  Further, the system $T$ has strongly dense periodic points.

(ii) 
When $|\lambda_1|\neq 1$ and $\lambda_2=1$, then the matrix
$$
A=\left({\begin{array}{cc}
\lambda_1&0\\
0&1\\
\end{array}}\right),
$$
where $|\lambda_1|\neq 1$. For all the point $x=\left({\begin{array}{cc}
0\\
x_2\\
\end{array}}\right)\in  \mathbb{R}^2$, where $x_2\in  \mathbb{R}$, we have $$T(x)=A(x)=\left({\begin{array}{cc}
\lambda_1&0\\
0&1\\
\end{array}}\right)\left({\begin{array}{cc}
0\\
x_2\\
\end{array}}\right)=\left({\begin{array}{cc}
0\\
x_2\\
\end{array}}\right)=x.$$ Therefore,all the point $x=\left({\begin{array}{cc}
0\\
x_2\\
\end{array}}\right)\in  \mathbb{R}^2$ are periodic points, where $x_2\in  \mathbb{R}$.  For any nonzero $x^*\in \mathbb{R}^{2*}$, for any $x\in  \mathbb{R}^2$ and for any $\varepsilon>0$, take periodic point $y=\left({\begin{array}{cc}
0\\
\frac{\varepsilon}{2|x^*|}\\
\end{array}}\right)$, then it is easy to see that $p_{x^*}(x-y)<\varepsilon$, therefore, the system $T$ has strongly dense periodic points.

When $|\lambda_1|\neq 1$ and $\lambda_2=-1$, then the matrix
$$
A=\left({\begin{array}{cc}
\lambda_1&0\\
0&-1\\
\end{array}}\right),
$$
where $|\lambda_1|\neq 1$. For all the point $x=\left({\begin{array}{cc}
0\\
x_2\\
\end{array}}\right)\in  \mathbb{R}^2$, where $x_2\in  \mathbb{R}$, we have $$T^2(x)=A^2(x)=\left({\begin{array}{cc}
\lambda_1&0\\
0&1\\
\end{array}}\right)\left({\begin{array}{cc}
0\\
x_2\\
\end{array}}\right)=\left({\begin{array}{cc}
0\\
x_2\\
\end{array}}\right)=x.$$ Therefore,all the point $x=\left({\begin{array}{cc}
0\\
x_2\\
\end{array}}\right)\in  \mathbb{R}^2$ are periodic points, where $x_2\in  \mathbb{R}$. With the similar argument, the system $T$ has strongly dense periodic points.

For the each point $x=\left({\begin{array}{cc}
x_1\\
x_2\\
\end{array}}\right)\in  \mathbb{R}^2$, where $x_1\neq 0$, and for any   $n\in \mathbb{N}$, we have $$T^n(x)=A^n(x)=\left({\begin{array}{cc}
\lambda_1^n&0\\
0&\lambda_2^n\\
\end{array}}\right)\left({\begin{array}{cc}
x_1\\
x_2\\
\end{array}}\right)=\left({\begin{array}{cc}
\lambda_1^nx_1\\
\lambda_2^nx_2\\
\end{array}}\right)\neq x.$$ Therefore, all the point $x=\left({\begin{array}{cc}
x_1\\
x_2\\
\end{array}}\right)\in  \mathbb{R}^2$ are not periodic points,  where $x_1\neq 0$.

(iii) With the similar argument of (ii), we can get the result.

(iv) If $|\lambda_1|\neq1$ and $|\lambda_2|\neq 1$, then for any  $x=\left({\begin{array}{cc}
x_1\\
x_2\\
\end{array}}\right)\in  \mathbb{R}^2$ with $x_1^2+x_2^2\neq0$ and any $n\in \mathbb{N}$, we have 

 $T^n(x)=A^n(x)=\left({\begin{array}{cc}
\lambda_1^n&0\\
0&\lambda_2^n\\
\end{array}}\right)x=\left({\begin{array}{cc}
\lambda_1^nx_1\\
\lambda_2^nx_2\\
\end{array}}\right)\neq x$. Therefore, $x$ is not periodic point. Further, the system $T$ does not have weakly dense periodic points. Besides, it is obvious that $x=\left({\begin{array}{cc}
0\\
0\\
\end{array}}\right)\in  \mathbb{R}^2$ is periodic point.

\end{proof}

\begin{Theorem}\label{th6}
Suppose that
$$
A=\left({\begin{array}{cc}
\lambda&1\\
0&\lambda\\
\end{array}}\right),
$$
where $\lambda$ is a real number. Consider the dynamical system $Tx=Ax$ with $x\in \mathbb{R}^2$.

(i) If $|\lambda|=1$, then all points $x=\left({\begin{array}{cc}
x_1\\
x_2\\
\end{array}}\right)\in \mathbb{R}^2$ with $x_1^2+x_2^2\neq0$ are not periodic points. Further, the system $T$ does not have weakly dense periodic points.

(ii) If  $|\lambda|\neq1$, then the point $x=\left({\begin{array}{cc}
x_1\\
0\\
\end{array}}\right)\in\mathbb{R}^2$ is periodic point, where  $x_1\in \mathbb{R}$, and the other points are not periodic points. Further, the system $T$ has strongly dense periodic points.

\end{Theorem}

\begin{proof}
By direct calculation, one obtains
$$
A^n=\left({\begin{array}{cc}
\lambda^n&n\lambda^{n-1}\\
0&\lambda^n\\
\end{array}}\right).
$$
For any $x=\left({\begin{array}{cc}
x_1\\
x_2\\
\end{array}}\right)\in \mathbb{R}^2$, $T^n(x)=A^n(x)=\left({\begin{array}{cc}
\lambda^nx_1+n\lambda^{n-1}x_2\\
\lambda^nx_2\\
\end{array}}\right)$.

(i) If $|\lambda|\neq1$, then for any $x=\left({\begin{array}{cc}
x_1\\
x_2\\
\end{array}}\right)\in\mathbb{R}^2$ with $x_1^2+x_2^2\neq0$ and any $n\in \mathbb{N}$, we have $T^n(x)=A^n(x)=\left({\begin{array}{cc}
\lambda^nx_1+n\lambda^{n-1}x_2\\
\lambda^nx_2\\
\end{array}}\right)\neq x$. Therefore, all points of $\mathbb{R}^2$ are not periodic points except for $\left({\begin{array}{cc}
0\\
0\\
\end{array}}\right)$.  Further, it is obvious that the system $T$ does not have weakly dense periodic points.

(ii)If $|\lambda|=1$, then for any $x=\left({\begin{array}{cc}
x_1\\
x_2\\
\end{array}}\right)\in\mathbb{R}^2$ with $x_2\neq0$ and any $n\in \mathbb{N}$, we have $T^n(x)=A^n(x)=\left({\begin{array}{cc}
\lambda^nx_1+n\lambda^{n-1}x_2\\
\lambda^nx_2\\
\end{array}}\right)=\left({\begin{array}{cc}
x_1+nx_2\\
x_2\\
\end{array}}\right)\neq x$. And for any $x=\left({\begin{array}{cc}
x_1\\
x_2\\
\end{array}}\right)\in \mathbb{R}^2$ with $x_2=0$, we have $T(x)=A(x)=\left({\begin{array}{cc}
\lambda^nx_1+n\lambda^{n-1}x_2\\
\lambda^nx_2\\
\end{array}}\right)=\left({\begin{array}{cc}
x_1\\
0\\
\end{array}}\right)=x$. Therefore, then all points of $\mathbb{R}^2$ are not periodic points except for $\left({\begin{array}{cc}
u\\
0\\
\end{array}}\right)$, where $u\in \mathbb{R}$. For any nonzero $x^*\in \mathbb{R}^{2*}$, for any $x\in  \mathbb{R}^2$ and for any $\varepsilon>0$, take periodic point $y=\left({\begin{array}{cc}
\frac{\varepsilon}{2|x^*|}\\
0\\
\end{array}}\right)$, then it is easy to see that $p_{x^*}(x-y)<\varepsilon$, therefore, the system $T$ has strongly dense periodic points.

\end{proof}

\begin{Theorem}\label{th7}
Suppose that
$$
A=\left({\begin{array}{cc}
cos2\pi\theta&-sin2\pi\theta\\
sin2\pi\theta&cos2\pi\theta\\
\end{array}}\right)\left({\begin{array}{cc}
r&0\\
0&r\\
\end{array}}\right),
$$
where $r$ is a real number and $\theta\in[0,1]$. Consider the dynamical system $Tx=Ax$ with $x\in \mathbb{R}^2$.

(i) If $\theta$ is a rational number and $r=1$, then  all points of $\mathbb{R}^2$ are periodic points. Further, the system $T$ has strongly dense periodic points.

(ii) If $\theta$ is a irrational number or $r\neq1$, then all points of $\mathbb{R}^2$ are not periodic points except for $\left({\begin{array}{cc}
0\\
0\\
\end{array}}\right)$. Further, the system $T$ does not have weakly dense periodic points.
\end{Theorem}

\begin{proof}
(i) By direct calculation, one obtains
$$
A^n=\left({\begin{array}{cc}
cos2\pi n\theta&-sin2\pi n\theta\\
sin2\pi n\theta&cos2\pi n\theta\\
\end{array}}\right)\left({\begin{array}{cc}
r^n&0\\
0&r^n\\
\end{array}}\right).
$$
Since $\theta$ is a rational number and $r=1$, there exists $k\in \mathbb{N}$ such that $cos2\pi k\theta=1$ and $sin2\pi k\theta=0$. Then
$$
A^k=\left({\begin{array}{cc}
1&0\\
0&1\\
\end{array}}\right)\left({\begin{array}{cc}
1&0\\
0&1\\
\end{array}}\right)=\left({\begin{array}{cc}
1&0\\
0&1\\
\end{array}}\right).
$$
For any $x=\left({\begin{array}{cc}
x_1\\
x_2\\
\end{array}}\right)\in R^2$, $T^k(x)=A^k(x)=\left({\begin{array}{cc}
x_1\\
x_2\\
\end{array}}\right)=x$.
Therefore,  all points of $R^2$ are periodic points. Further, the system $T$ has strongly dense periodic points.

(ii) If $r\neq1$, then for any $x=\left({\begin{array}{cc}
x_1\\
x_2\\
\end{array}}\right)\in \mathbb{R}^2$ with $x_1^2+x_2^2\neq0$ and any $n\in \mathbb{N}$, $T^n(x)=A^n(x)=\left({\begin{array}{cc}
cos2\pi n\theta&-sin2\pi n\theta\\
sin2\pi n\theta&cos2\pi n\theta\\
\end{array}}\right)\left({\begin{array}{cc}
r^n&0\\
0&r^n\\
\end{array}}\right)\left({\begin{array}{cc}
x_1\\
x_2\\
\end{array}}\right)\neq x$.

If $r=1$ but $\theta$ is a irrational number, then for any $n\in \mathbb{N}$,
\begin{equation*}
\begin{aligned}
A^n(x)&=\left({\begin{array}{cc}
cos2\pi n\theta&-sin2\pi n\theta\\
sin2\pi n\theta&cos2\pi n\theta\\
\end{array}}\right)\left({\begin{array}{cc}
r^n&0\\
0&r^n\\
\end{array}}\right)\left({\begin{array}{cc}
x_1\\
x_2\\
\end{array}}\right)\\
&=\left({\begin{array}{cc}
cos2\pi n\theta&-sin2\pi n\theta\\
sin2\pi n\theta&cos2\pi n\theta\\
\end{array}}\right)\left({\begin{array}{cc}
1&0\\
0&1\\
\end{array}}\right) \left({\begin{array}{cc}
x_1\\
x_2\\
\end{array}}\right)\\
&=\left({\begin{array}{cc}
cos2\pi n\theta&-sin2\pi n\theta\\
sin2\pi n\theta&cos2\pi n\theta\\
\end{array}}\right)\left({\begin{array}{cc}
x_1\\
x_2\\
\end{array}}\right)\\
&\neq \left({\begin{array}{cc}
1&0\\
0&1\\
\end{array}}\right)\left({\begin{array}{cc}
x_1\\
x_2\\
\end{array}}\right)\\
&=x.
\end{aligned}
\end{equation*}
Therefore, for any $x=\left({\begin{array}{cc}
x_1\\
x_2\\
\end{array}}\right)\in \mathbb{R}^2$ with $x_1^2+x_2^2\neq0$ and any $n\in \mathbb{N}$, $T^n(x)=A^n(x)\neq x$.
Therefore, all points of $\mathbb{R}^2$ are not periodic points except for $\left({\begin{array}{cc}
0\\
0\\
\end{array}}\right)$. Further, the system $T$ does not have weakly dense periodic points.

\end{proof}

By the Theorem \ref{th5}, Theorem \ref{th6} and Theorem \ref{th7}, we have the following  corollary.
\begin{Corollary}
Consider the dynamical system $Tx=Ax$, where $x\in\mathbb{R}^2$ and $A$ is a $2\times 2$  matrix.

Then $T$ has strongly dense periodic points in weak topology if and only if $T$ has weakly dense periodic points

\end{Corollary}

\section{Sensitive in weakly topology}
  This section focuses on the investigation of sensitivity in weakly topology in two-dimensional Euclidean spaces.

	\begin{Theorem}\label{ths1}
	Suppose that 
	$\begin{pmatrix}
		\lambda_1 & 0 \\
		0 & \lambda_2
	\end{pmatrix}$,
where $\lambda_1$ and $\lambda_2$ are real numbers. Consider the dynamical system $Tx = Ax$ with $x\in\mathbb{R}^2$.

$(i)$ If $\left|\lambda_1\right|>1$ and $\left|\lambda_2\right|>1$, then $T$ is strongly sensitive in weakly topology.

$(ii)$ If else, then $T$ is not strongly sensitive in weakly topology.

  \end{Theorem}
\begin{proof}
	case $(i)$:
	 Take $\delta=1$. For any $x^{*}=\left({\begin{array}{cc}
x_{1}^{*}\\
x_{2}^{*}\\
\end{array}}\right)\in \mathbb{R}^{2*}\backslash\{0\}$, any $x_0=\left({\begin{array}{cc}
x_0^1\\
x_0^2\\
\end{array}}\right)\in\mathbb{R}^{2}$ and any $\epsilon>0$,  Without loss of generality, assume that $x_{1}^{*}\neq 0$, then take $x=\left({\begin{array}{cc}
x_0^1+\frac{\epsilon}{2x_{1}^{*}}\\
x_0^2\\
\end{array}}\right),y=x_0=\left({\begin{array}{cc}
x_0^1\\
x_0^2\\
\end{array}}\right)\in \mathbb{R}^2$, then
 we have 
$$p_{x^*}(x-x_0)=\left|(x_0^1+\frac{\epsilon}{2x_{1}^{*}}-x_0^1)x_{1}^{*}+(x_0^2-x_0^2)x_{2}^{*}\right|=\frac{\epsilon}{2}<\epsilon,$$
$$p_{x^*}(y-x_0)=0<\epsilon.$$
Besides, for any $n\in\mathbb{N}$
\begin{equation*}
\begin{aligned}
p_{x^*}(T^n(x)-T^n(y))&=|\langle(T^n(x)-T^n(y)),x^*\rangle| \\
&=|\langle A^n(x-y),x^*\rangle| \\
&=|\lambda_1^n(x_0^1+\frac{\epsilon}{2x_{1}^{*}}-x_0^1)x_1^*+\lambda_2^n(x_0^2-x_0^2)x_2^*| \\
&=|\lambda_1^n\frac{\epsilon}{2x_{1}^{*}}x_1^*| \\
&=|\lambda_1|^n\frac{\epsilon}{2}.
\end{aligned}
\end{equation*}
Since  $\left|\lambda_1\right|>1$, for the $\epsilon$ above, there exists $n\in\mathbb{N}$ large enough such that $|\lambda_1|^n\frac{\epsilon}{2}>1=\delta$. Therefore, $T$ is strongly sensitive in weakly topology.

 case $(ii)$:
    Without loss of generality, assume that $\left|\lambda_1\right|\le1$. For any $\delta>0$, take $x^{*}=\left({\begin{array}{cc}
x_{1}^{*}\\
x_{2}^{*}\\
\end{array}}\right)=\left({\begin{array}{cc}
1\\
0\\
\end{array}}\right) \in \mathbb{R}^{2*}\backslash\{0\}$, $x_0=\left({\begin{array}{cc}
1\\
0\\
\end{array}}\right)\in\mathbb{R}^{2}$ and $\epsilon=\frac{\delta}{4}$. Then for any $x=\left({\begin{array}{cc}
x_1\\
x_2\\
\end{array}}\right),y=\left({\begin{array}{cc}
y_1\\
y_2\\
\end{array}}\right)\in \mathbb{R}^2 $ with $p_{x^*}(x-x_0)<\epsilon$ and  $p_{x^*}(y-x_0)<\epsilon$, we have 
\begin{equation}\label{ths11}
p_{x^*}(x-x_0)=\left|x_1-1\right|<\epsilon=\frac{\delta}{4},
\end{equation}
\begin{equation}\label{ths12}
p_{x^*}(y-x_0)=\left|y_1-1\right|<\epsilon=\frac{\delta}{4},
\end{equation}
Combine (\ref{ths11}) and (\ref{ths12}), we have 
\begin{equation}\label{ths13}
\left|x_1-y_1\right|<\frac{\delta}{2}.
\end{equation}
On the other hand, for any $n\in\mathbb{N}$, 
\begin{equation}\label{ths14}
\begin{aligned}
p_{x^*}(T^n(x)-T^n(y))&=|\langle(T^n(x)-T^n(y)),x^*\rangle| \\
&=|\langle A^n(x-y),x^*\rangle| \\
&=|\lambda_1^n(x_1-y_1)x_1^*+\lambda_2^n(x_2-y_2)x_2^*| \\
&=|\lambda_1|^n|x_1-y_1|\\
&\le|x_1-y_1|,
\end{aligned}
\end{equation}
where the last inequality follows since  $\left|\lambda_1\right|\le1$.
Combine (\ref{ths13}) and (\ref{ths14}), we have 
$$p_{x^*}(T^n(x)-T^n(y))<\delta,$$
which shows that $T$ is not strongly sensitive in weakly topology.

\end{proof}

	\begin{Theorem}\label{ths2}
	Suppose that 
	$\begin{pmatrix}
		\lambda_1 & 0 \\
		0 & \lambda_2
	\end{pmatrix}$,
where $\lambda_1$ and $\lambda_2$ are real numbers. Consider the dynamical system $Tx = Ax$ with $x\in\mathbb{R}^2$.

$(i)$ If $\lambda_1=\lambda_2$ and $\left|\lambda_1\right|\le1$, then $T$ is not weakly sensitive in weakly topology.

$(ii)$ If else, then $T$ is weakly sensitive in weakly topology.

  \end{Theorem}
\begin{proof}
	case $(i)$:
	 For any $\delta>0$ and any $x^{*}=\left({\begin{array}{cc}
x_{1}^{*}\\
x_{2}^{*}\\
\end{array}}\right)\in \mathbb{R}^{2*}\backslash\{0\}$, take $x_0=\left({\begin{array}{cc}
0\\
0\\
\end{array}}\right)\in\mathbb{R}^{2}$ and $\epsilon=\frac{\delta}{4}$. For any  $x=\left({\begin{array}{cc}
x_1\\
x_2\\
\end{array}}\right),y=\left({\begin{array}{cc}
y_1\\
y_2\\
\end{array}}\right)\in \mathbb{R}^2 $ with $p_{x^*}(x-x_0)<\epsilon$ and  $p_{x^*}(y-x_0)<\epsilon$, we have 
\begin{equation}\label{ths21}
p_{x^*}(x-x_0)<\epsilon=\frac{\delta}{4},
\end{equation}
\begin{equation}\label{ths22}
p_{x^*}(y-x_0)<\epsilon=\frac{\delta}{4},
\end{equation}
Combine (\ref{ths21}) and (\ref{ths22}), we have 
\begin{equation}\label{ths23}
p_{x^*}(x-y)\le p_{x^*}(x)+p_{x^*}(y)\le \frac{\delta}{2}.
\end{equation}

On the other hand, for any $n\in\mathbb{N}$, since  $\lambda_1=\lambda_2$, we have
\begin{equation}\label{ths24}
\begin{aligned}
p_{x^*}(T^n(x)-T^n(y))&=|\langle(T^n(x)-T^n(y)),x^*\rangle| \\
&=|\langle A^n(x-y),x^*\rangle| \\
&=|\lambda_1^n(x_1-y_1)x_1^*+\lambda_2^n(x_2-y_2)x_2^*| \\
&=|\lambda_1^n(x_1-y_1)x_1^*+\lambda_1^n(x_2-y_2)x_2^*| \\
&=|\lambda_1|^n|(x_1-y_1)x_1^*+(x_2-y_2)x_2^*|\\
&=|\lambda_1|^n|(x_1-y_1)x_1^*+(x_2-y_2)x_2^*|\\
&=|\lambda_1|^np_{x^*}(x-y).
\end{aligned}
\end{equation}
By the assumption that $\left|\lambda_1\right|\le1$,  (\ref{ths23}) and (\ref{ths24}), we have 
$$p_{x^*}(T^n(x)-T^n(y))<\delta$$ 
for any any $n\in\mathbb{N}$. Therefore, $T$ is not weakly sensitive in weakly topology.

 case $(ii)$: if $\left|\lambda_1\right|>1$, then take $\delta=1$ and $x^{*}=\left({\begin{array}{cc}
1\\
0\\
\end{array}}\right)\in \mathbb{R}^{2*}\backslash\{0\}$. For any  $x_0=\left({\begin{array}{cc}
x_0^1\\
x_0^2\\
\end{array}}\right)\in\mathbb{R}^{2}$ and any $\epsilon>0$, take $x=\left({\begin{array}{cc}
x_1\\
x_2\\
\end{array}}\right)=\left({\begin{array}{cc}
x_0^1+\frac{\epsilon}{2}\\
x_0^2\\
\end{array}}\right),y=\left({\begin{array}{cc}
y_1\\
y_2\\
\end{array}}\right)=x_0=\left({\begin{array}{cc}
x_0^1\\
x_0^2\\
\end{array}}\right)$, then
 $$p_{x^*}(x-x_0)=\left|\frac{\epsilon}{4}\right|<\epsilon,$$
$$p_{x^*}(y-x_0)=0<\epsilon.$$
Besides, for any $n\in\mathbb{N}$,
\begin{equation}
\begin{aligned}
p_{x^*}(T^n(x)-T^n(y))&=|\langle(T^n(x)-T^n(y)),x^*\rangle| \\
&=|\langle A^n(x-y),x^*\rangle| \\
&=|\lambda_1^n(\frac{\epsilon}{2}+0| \\
&=|\lambda_1|^n\frac{\epsilon}{2}.
\end{aligned}
\end{equation}
Since  $\left|\lambda_1\right|>1$, for the $\epsilon$ above, there exists $n\in\mathbb{N}$ large enough such that $|\lambda_1|^n\frac{\epsilon}{2}>1=\delta$. Therefore, $T$ is weakly sensitive in weakly topology.

If $\left|\lambda_2\right|>1$, with the similar argument, $T$ is weakly sensitive in weakly topology.

Next, we will only consider that $\lambda_1\neq\lambda_2$ with $\left|\lambda_1\right|\le1$ and $\left|\lambda_2\right|\le1$.

 Without loss of generality, assume that $\lambda_1>\lambda_2$. Take $\delta=1$ and $x^{*}=\left({\begin{array}{cc}
1\\
1\\
\end{array}}\right)\in \mathbb{R}^{2*}\backslash\{0\}$. For any  $x_0=\left({\begin{array}{cc}
x_0^1\\
x_0^2\\
\end{array}}\right)\in\mathbb{R}^{2}$ and any $\epsilon>0$, take $x=\left({\begin{array}{cc}
x_1\\
x_2\\
\end{array}}\right)=\left({\begin{array}{cc}
x_0^1+\frac{4}{\lambda_1-\lambda_2}\\
x_0^2-\frac{4}{\lambda_1-\lambda_2}\\
\end{array}}\right),y=\left({\begin{array}{cc}
y_1\\
y_2\\
\end{array}}\right)=x_0=\left({\begin{array}{cc}
x_0^1\\
x_0^2\\
\end{array}}\right)$, then
\begin{equation}
\begin{aligned}
p_{x^*}(x-x_0)&=\left|(x_0^1+\frac{4}{\lambda_1-\lambda_2}-x_0^1)x_{1}^{*}+(x_0^2-\frac{4}{\lambda_1-\lambda_2}-x_0^2)x_{2}^{*}\right|\\
&=\left|(x_0^1+\frac{4}{\lambda_1-\lambda_2}-x_0^1)+(x_0^2-\frac{4}{\lambda_1-\lambda_2}-x_0^2)\right|\\
&=0<\epsilon
\end{aligned}
\end{equation}
and
$$p_{x^*}(y-x_0)=0<\epsilon.$$

Besides, take $n=1$,
\begin{equation}
\begin{aligned}
&p_{x^*}(T^n(x)-T^n(y))\\
=&p_{x^*}(T(x)-T(y)) \\
=&|\langle(T(x)-T(y)),x^*\rangle| \\
=&|\langle A(x-y),x^*\rangle| \\
=&|\lambda_1(x_0^1+\frac{4}{\lambda_1-\lambda_2}-x_0^1)x_{1}^{*}+\lambda_2(x_0^2-\frac{4}{\lambda_1-\lambda_2}-x_0^2)x_{2}^{*}| \\
=&|\lambda_1(\frac{4}{\lambda_1-\lambda_2})+\lambda_2(-\frac{4}{\lambda_1-\lambda_2})| \\
=&4\\
>&1.
\end{aligned}
\end{equation}
 Therefore, $T$ is weakly sensitive in weakly topology.

\end{proof}

\begin{Theorem}\label{ths3}
	Suppose that 
		\[
	\begin{pmatrix}
		\lambda & 1 \\
		0 & \lambda
	\end{pmatrix},
	\]
	where $\lambda$ is a real number. Consider the dynamical system $Tx=Ax$ for $x\in\mathbb{R}^{2}$.

$(i)$ If $|\lambda|>1$, then $T$ is strongerly sensitive in weakly topology.

$(ii)$ If else, then $T$ is not strongerly sensitive in weakly topology. 
\end{Theorem}
\begin{proof}
Case $(i)$: 
 Take $\delta=1$. For any $x^{*}=\left({\begin{array}{cc}
x_{1}^{*}\\
 x_{2}^{*}\\
\end{array}}\right)\in \mathbb{R}^{2*}\backslash\{0\}$, any $x_0=\left({\begin{array}{cc}
x_0^1\\
x_0^2\\
\end{array}}\right)\in\mathbb{R}^{2}$ and any $\epsilon>0$. If $x_{1}^{*}\neq 0$, then take $x=\left({\begin{array}{cc}
x_0^1+\frac{\epsilon}{2x_{1}^{*}}\\
x_0^2\\
\end{array}}\right)$ and $y=x_0=\left({\begin{array}{cc}
x_0^1\\
x_0^2\\
\end{array}}\right)$, then
 $$p_{x^*}(x-x_0)=\left|(x_0^1+\frac{\epsilon}{2x_{1}^{*}}-x_0^1)x_{1}^{*}+(x_0^2-x_0^2)x_{2}^{*}\right|=\frac{\epsilon}{2}<\epsilon,$$
$$p_{x^*}(y-x_0)=0<\epsilon.$$
Besides, for any $n\in\mathbb{N}$,
\begin{equation*}
\begin{aligned}
p_{x^*}(T^n(x)-T^n(y))&=|\langle(T^n(x)-T^n(y)),x^*\rangle| \\
&=|\langle A^n(x-y),x^*\rangle| \\
&=|\lambda^n(x_0^1+\frac{\epsilon}{2x_{1}^{*}}-x_0^1)x_1^*+n\lambda^{n-1}(x_0^2-x_0^2)x_1^*+\lambda^n(x_0^2-x_0^2)x_2^*| \\
&=|\lambda^n\frac{\epsilon}{2x_{1}^{*}}x_1^*| \\
&=|\lambda|^n\frac{\epsilon}{2}.
\end{aligned}
\end{equation*}
Since  $\left|\lambda\right|>1$, for the $\epsilon$ above, there exists $n\in\mathbb{N}$ large enough such that $|\lambda|^n\frac{\epsilon}{2}>1=\delta$. Therefore, $T$ is strongly sensitive in weakly topology.

If $x_{2}^{*}\neq 0$, then take $x=\left({\begin{array}{cc}
x_0^1\\
x_0^2+\frac{\epsilon}{2x_{2}^{*}}\\
\end{array}}\right)$ and $y=x_0=\left({\begin{array}{cc}
x_0^1\\
x_0^2\\
\end{array}}\right)$, then
 $$p_{x^*}(x-x_0)=\left|(x_0^1-x_0^1)x_{1}^{*}+(x_0^2+\frac{\epsilon}{2x_{2}^{*}}-x_0^2)x_{2}^{*}\right|=\frac{\epsilon}{2}<\epsilon,$$
$$p_{x^*}(y-x_0)=0<\epsilon.$$
Besides, for any $n\in\mathbb{N}$,
\begin{equation*}
\begin{aligned}
&p_{x^*}(T^n(x)-T^n(y))\\
=&|\langle(T^n(x)-T^n(y)),x^*\rangle| \\
=&|\langle A^n(x-y),x^*\rangle| \\
=&|\lambda^n(x_0^1-x_0^1)x_1^*+n\lambda^{n-1}(x_0^2+\frac{\epsilon}{2x_{2}^{*}}-x_0^2)x_1^*+\lambda^n(x_0^2+\frac{\epsilon}{2x_{2}^{*}}-x_0^2)x_2^*| \\
=&|n\lambda^{n-1}\frac{\epsilon}{2x_{2}^{*}}x_1^*+\lambda^n\frac{\epsilon}{2x_{2}^{*}}x_2^*| \\
=&|n\lambda^{n-1}\frac{\epsilon x_{1}^{*}}{2x_{2}^{*}}+\lambda^n\frac{\epsilon}{2}|.
\end{aligned}
\end{equation*}
Since  $\left|\lambda\right|>1$, for the $\epsilon$ above, there exists $n\in\mathbb{N}$ large enough such that $|n\lambda^{n-1}\frac{\epsilon x_{1}^{*}}{2x_{2}^{*}}+\lambda^n\frac{\epsilon}{2}|>1=\delta$. Therefore, $T$ is strongly sensitive in weakly topology.

 case $(ii)$:
  For any $\delta>0$, take $x^{*}=\left({\begin{array}{cc}
0\\
1\\
\end{array}}\right)\in \mathbb{R}^{2*}\backslash\{0\}$, $x_0=\left({\begin{array}{cc}
0\\
0\\
\end{array}}\right)\in\mathbb{R}^{2}$ and $\epsilon=\frac{\delta}{4}$. Then for any $x=\left({\begin{array}{cc}
x_1\\
x_2\\
\end{array}}\right),y=\left({\begin{array}{cc}
y_1\\
y_2\\
\end{array}}\right)\in \mathbb{R}^2 $ with $p_{x^*}(x-x_0)<\epsilon$ and  $p_{x^*}(y-x_0)<\epsilon$, we have 
\begin{equation}\label{ths41}
p_{x^*}(x-x_0)=\left|x_2\right|<\epsilon=\frac{\delta}{4},
\end{equation}
\begin{equation}\label{ths42}
p_{x^*}(y-x_0)=\left|y_2\right|<\epsilon=\frac{\delta}{4},
\end{equation}
Combine (\ref{ths41}) and (\ref{ths42}), we have 
\begin{equation}\label{ths43}
\left|x_2-y_2\right|<\frac{\delta}{2}.
\end{equation}
On the other hand, for any $n\in\mathbb{N}$, 
\begin{equation}\label{ths44}
\begin{aligned}
p_{x^*}(T^n(x)-T^n(y))&=|\langle(T^n(x)-T^n(y)),x^*\rangle| \\
&=|\langle A^n(x-y),x^*\rangle| \\
&=|\lambda^n(x_1-y_1)x_1^*+n\lambda^{n-1}(x_2-y_2)x_1^*+\lambda^n(x_2-y_2)x_2^*| \\
&=|\lambda|^n|x_2-y_2|\\
&\le|x_2-y_2|.
\end{aligned}
\end{equation}
Combine (\ref{ths43}) and (\ref{ths44}), we have 
$$p_{x^*}(T^n(x)-T^n(y))<\delta,$$
which shows that $T$ is not strongly sensitive in weakly topology.

\end{proof}

\begin{Theorem}\label{ths4}
	Suppose that 
		\[
	\begin{pmatrix}
		\lambda & 1 \\
		0 & \lambda
	\end{pmatrix},
	\]
	where $\lambda$ is a real number. Consider the dynamical system $Tx=Ax$ for $x\in\mathbb{R}^{2}$.

$(i)$ If $|\lambda|>0$, then $T$ is weakly sensitive in weakly topology.

$(ii)$ If $\lambda=0$, then $T$ is not weakly sensitive in weakly topology. 
\end{Theorem}
\begin{proof}
Case $(i)$: Take $\delta=1$ $x^{*}=\left({\begin{array}{cc}
1\\
0\\
\end{array}}\right)\in \mathbb{R}^{2*}\backslash\{0\}$. For any  $x_0=\left({\begin{array}{cc}
x_0^1\\
x_0^2\\
\end{array}}\right)\in\mathbb{R}^{2}$ and any $\epsilon>0$ with $\epsilon<\frac{1}{|\lambda|}$, take $x=\left({\begin{array}{cc}
x_1\\
x_2\\
\end{array}}\right)=\left({\begin{array}{cc}
x_0^1\\
x_0^2+2\\
\end{array}}\right),y=\left({\begin{array}{cc}
y_1\\
y_2\\
\end{array}}\right)=x_0=\left({\begin{array}{cc}
x_0^1\\
x_0^2\\
\end{array}}\right)$, then
 $$p_{x^*}(x-x_0)=0<\epsilon,$$
$$p_{x^*}(x-x_0)=0<\epsilon.$$
Besides, for any $n\in\mathbb{N}$, 
\begin{equation}
\begin{aligned}
p_{x^*}(T^n(x)-T^n(y))&=|\langle(T^n(x)-T^n(y)),x^*\rangle| \\
&=|\langle A^n(x-y),x^*\rangle| \\
&=|\lambda^n(x_0^1-y_0^1)x_1^*+n\lambda^{n-1}(x_0^2-y_0^2)x_1^*+\lambda^n(x_0^2-y_0^2)x_2^*| \\
&=|2n\lambda^{n-1}|.
\end{aligned}
\end{equation}
Take $n=1$, then 
$$p_{x^*}(T^n(x)-T^n(y))=|2n\lambda^{n-1}|=2>1.$$
Therefore, $T$ is weakly sensitive in weakly topology.

 case $(ii)$:
  If $\lambda=0$, then for any $n\in\mathbb{N}$,
\begin{equation*}
\begin{aligned}
p_{x^*}(T^n(x)-T^n(y))&=|\langle(T^n(x)-T^n(y)),x^*\rangle| \\
&=|\langle A^n(x-y),x^*\rangle| \\
&=|\lambda^n(x_0^1-y_0^1)x_1^*+n\lambda^{n-1}(x_0^2-y_0^2)x_1^*+\lambda^n(x_0^2-y_0^2)x_2^*| \\
&=0.
\end{aligned}
\end{equation*}
Therefore, $T$ is not weakly sensitive in weakly topology.

\end{proof}

\begin{Theorem}\label{ths5}
	Consider the dynamical system $Tx=Ax$, where $x\in\mathbb{R}^2$ and $A$ is a $2\times 2$ real matrix with two conjugate complex eigenvalues,
	\[\begin{pmatrix}
		cos(2\pi\theta) & -sin(2\pi\theta) \\
		sin(2\pi\theta) & cos(2\pi\theta)
	\end{pmatrix}
	\begin{pmatrix}
		r & 0 \\
		0 & r
	\end{pmatrix}
	\]
where $r>0$ and $\theta\in[0,1]$.

$(i)$ If $\theta\notin\{0,0.5,1\}$, then $T$ is strongly sensitive in weakly topology.

$(ii)$ If $\theta\in\{0,0.5,1\}$ and $r>1$, then $T$ is strongly sensitive in weakly topology.

$(iii)$ If  $\theta\in\{0,0.5,1\}$ and $r\le 1$, then $T$ is not weakly sensitive in weakly topology.

\end{Theorem}
\begin{proof}

Case $(i)$: 
 Take $\delta=1$. For any $x^{*}=\left({\begin{array}{cc}
x_{1}^{*}\\
x_{2}^{*}\\
\end{array}}\right)\in \mathbb{R}^{2*}\backslash\{0\}$, any $x_0=\left({\begin{array}{cc}
x_0^1\\
x_0^2\\
\end{array}}\right)\in\mathbb{R}^{2}$ and any $\epsilon>0$. Take $y=\left({\begin{array}{cc}
y_1\\
y_2\\
\end{array}}\right)=x_0=\left({\begin{array}{cc}
x_0^1\\
x_0^2\\
\end{array}}\right)\in\mathbb{R}^{2}$. Since $\theta\notin\{0,0.5,1\}$, there exists $x=\left({\begin{array}{cc}
x_1\\
x_2\\
\end{array}}\right)\in\mathbb{R}^{2}$ such that $\overrightarrow{x_0x}\perp\overrightarrow{x^*o}$ and 
\begin{equation}\label{th61}
\left|\langle B(x-y),x^{*}\rangle\right|>\frac{1}{r},
\end{equation}
where 
$$B=\begin{pmatrix}
		                	cos(2\pi\theta) & -sin(2\pi\theta) \\
		                	sin(2\pi\theta) & cos(2\pi\theta)
		                \end{pmatrix}
		                .$$
So, take $n=1$, we have 
	               \begin{align*}
P_{x^{*}}(T^{n}x-T^{n}y)&=P_{x^{*}}(Tx-Ty)\\
&=\left|<A(x-y), x^*>\right|\\
		                &=\left|<r \begin{pmatrix}
		                	cos(2\pi \theta) & -sin(2\pi \theta) \\
		                	sin(2\pi \theta) & cos(2\pi \theta)
		                \end{pmatrix}
		                \begin{pmatrix}
		                	x_1-y_1\\
		                	x_2-y_2
		                \end{pmatrix},\begin{pmatrix}
		                x_1^*\\ x_2^*
		                \end{pmatrix}>\right|\\
		                &=r\left|\langle B(x-y),x^{*}\rangle\right|\\
		                &>1.
	\end{align*}
In the last step we used (\ref{th61}).
Therefore, $T$ is strongly sensitive in weakly topology.

 case $(ii)$:
Denote
$$A_\theta=\begin{pmatrix}
		                	cos(2\pi\theta) & -sin(2\pi\theta) \\
		                	sin(2\pi\theta) & cos(2\pi\theta)
		                \end{pmatrix}
		                .$$
By direct computation, one has 
 \begin{align*}
A^n&=\begin{pmatrix}
		                	cos(2\pi \theta) & -sin(2\pi \theta) \\
		                	sin(2\pi \theta) & cos(2\pi \theta)
		                \end{pmatrix}
		               \begin{pmatrix}
		                	r^n & 0 \\
		                	0 & r^n
		                \end{pmatrix}\\
		                &=r^n\begin{pmatrix}
		                	cos(2\pi n\theta) & -sin(2\pi n\theta) \\
		                	sin(2\pi n\theta) & cos(2\pi n\theta)
		                \end{pmatrix}\\
		                &=r^nA_{n\theta}.
	\end{align*}

 Since $\theta\in\{0,0.5,1\}$, we have 
\begin{equation}\label{th62}
\left|\langle A_{n\theta}(x-y),x^{*}\rangle\right|=\left|\langle x-y,x^{*}\rangle\right|
\end{equation}
 for any $n\in\mathbb{N}$.
Take $\delta=1$. For any $x^{*}=\left({\begin{array}{cc}
x_{1}^{*}\\
 x_{2}^{*}\\
\end{array}}\right)\in \mathbb{R}^{2*}\backslash\{0\}$, any $x_0=\left({\begin{array}{cc}
x_0^1\\
x_0^2\\
\end{array}}\right)\in\mathbb{R}^{2}$ and any $\epsilon>0$. Take $y=\left({\begin{array}{cc}
y_1\\
y_2\\
\end{array}}\right)=x_0=\left({\begin{array}{cc}
x_0^1\\
x_0^2\\
\end{array}}\right)\in\mathbb{R}^{2}$. And there exists $x=\left({\begin{array}{cc}
x_{1}^{*}\\
 x_{2}^{*}\\
\end{array}}\right)(x_1,x_2)\in\mathbb{R}^{2}$ such that $0<\left|\langle x-x_0,x^{*}\rangle\right|<\epsilon$. 
So 
\begin{align*}
p_{x^{*}}(T^{n}x-T^{n}y)&=\left|<A^{n}(x-y), x^*>\right|\\
		                &=\left|<r^n \begin{pmatrix}
		                	cos(2\pi n\theta) & -sin(2\pi n\theta) \\
		                	sin(2\pi n\theta) & cos(2\pi n\theta)
		                \end{pmatrix}
		                \begin{pmatrix}
		                	x_1-y_1\\
		                	x_2-y_2
		                \end{pmatrix},\begin{pmatrix}
		                x_1^*\\ x_2^*
		                \end{pmatrix}>\right|\\
		                &=\left|<r^nA_{n\theta}  \begin{pmatrix}
		                	x_1-y_1\\
		                	x_2-y_2
		                \end{pmatrix},\begin{pmatrix}
		                x_1^* \\
		                x_2^*
		                \end{pmatrix}>\right|\\
&=r^n\left|<A_{n\theta}  \begin{pmatrix}
		                	x_1-y_1\\
		                	x_2-y_2
		                \end{pmatrix},\begin{pmatrix}
		                x_1^* \\
		                x_2^*
		                \end{pmatrix}>\right|\\
&=r^n\left|\langle A_{n\theta}(x-y),x^{*}\rangle\right|\\
		                &=r^n\left|\langle x-y,x^{*}\rangle\right|,
	\end{align*}
where the last equality is due to (\ref{th62}).
Since $0<\left|\langle x-x_0,x^{*}\rangle\right|=\left|\langle x-y,x^{*}\rangle\right|<\epsilon$ and $r>1$, there exists  $n\in\mathbb{N}$ large enough such that $r^n\left|\langle x-y,x^{*}\rangle\right|>1=\delta$, that is, $p_{x^{*}}(T^{n}x-T^{n}y)>\delta$. Therefore, $T$ is strongly sensitive in weakly topology.

case $(iii)$:

 For any $\delta>0$ and any $x^{*}=\left({\begin{array}{cc}
x_{1}^{*}\\
x_{2}^{*}\\
\end{array}}\right)\in \mathbb{R}^{2*}\backslash\{0\}$, take $x_0=\left({\begin{array}{cc}
0\\
0\\
\end{array}}\right)\in\mathbb{R}^{2}$ and $\epsilon=\frac{\delta}{2}$. For any  $x=\left({\begin{array}{cc}
x_1\\
x_2\\
\end{array}}\right),y=\left({\begin{array}{cc}
y_1\\
y_2\\
\end{array}}\right)\in \mathbb{R}^2 $ with $p_{x^*}(x-x_0)<\epsilon$ and  $p_{x^*}(y-x_0)<\epsilon$, we have $p_{x^*}(x-y)<2\epsilon=\delta$.

On the other hand, for any $n\in\mathbb{N}$, 
\begin{align*}
p_{x^{*}}(T^{n}x-T^{n}y)&=\left|<A^{n}(x-y), x^*>\right|\\
		                &=\left|<r^n \begin{pmatrix}
		                	cos(2\pi n\theta) & -sin(2\pi n\theta) \\
		                	sin(2\pi n\theta) & cos(2\pi n\theta)
		                \end{pmatrix}
		                \begin{pmatrix}
		                	x_1-y_1\\
		                	x_2-y_2
		                \end{pmatrix},\begin{pmatrix}
		                x_1^*\\ x_2^*
		                \end{pmatrix}>\right|\\
		                &=\left|<r^nA_{n\theta}  \begin{pmatrix}
		                	x_1-y_1\\
		                	x_2-y_2
		                \end{pmatrix},\begin{pmatrix}
		                x_1^* \\
		                x_2^*
		                \end{pmatrix}>\right|\\
&=r^n\left|<A_{n\theta}  \begin{pmatrix}
		                	x_1-y_1\\
		                	x_2-y_2
		                \end{pmatrix},\begin{pmatrix}
		                x_1^* \\
		                x_2^*
		                \end{pmatrix}>\right|\\
&=r^n\left|\langle A_{n\theta}(x-y),x^{*}\rangle\right|\\
		                &=r^n\left|\langle x-y,x^{*}\rangle\right|.
	\end{align*}
By the assumption that $\left|r\right|\le1$ and that $\left|\langle x-y,x^{*}\rangle\right|=p_{x^*}(x-y)<2\epsilon=\delta$, we have 
$$p_{x^*}(T^n(x)-T^n(y))<\delta$$ 
for any any $n\in\mathbb{N}$. 
Therefore, $T$ is not weakly sensitive in weakly topology.

\end{proof}

By the Theorem \ref{ths5}, we have the following  corollary.
\begin{Corollary}
Consider the dynamical system $Tx=Ax$, where $x\in\mathbb{R}^2$ and $A$ is a $2\times 2$ real matrix with two conjugate complex eigenvalues,
	\[\begin{pmatrix}
		cos(2\pi\theta) & -sin(2\pi\theta) \\
		sin(2\pi\theta) & cos(2\pi\theta)
	\end{pmatrix}
	\begin{pmatrix}
		r & 0 \\
		0 & r
	\end{pmatrix}
	\]
where $r>0$ and $\theta\in[0,1]$.

Then $T$ is strongly sensitive in weak topology if and only if $T$ is weakly sensitive in weak topology.

\end{Corollary}

\section{Devaney chaotic in weakly topology}
  This section focuses on the investigation of Devaney chaos in weakly topology in two-dimensional Euclidean spaces.

\begin{Theorem}
Suppose that
$$
A=\left({\begin{array}{cc}
\lambda_1&0\\
0&\lambda_2\\
\end{array}}\right),
$$
where $\lambda_1$ and $\lambda_2$ are real numbers. Consider the dynamical system $Tx=Ax$ with $x\in R^2$.
Then the system $T$ is not strongly Devaney chaos in weak topology. Besides, if $\lambda_1\neq\lambda_2$ with $|\lambda_1|=1$ or $|\lambda_2|=1$, then the system $T$ is weakly Devaney chaos in weak topology. If else, the system $T$ is not weakly Devaney chaos in weak topology.

\end{Theorem}

\begin{proof}
By Theorem \ref{th2}, Theorem \ref{th5}, Theorem \ref{ths1} and  Theorem \ref{ths2}, we can get the result.

\end{proof}

\begin{Theorem}
Suppose that
$$
A=\left({\begin{array}{cc}
\lambda&1\\
0&\lambda\\
\end{array}}\right),
$$
where $\lambda$ is a real number. Consider the dynamical system $Tx=Ax$ with $x\in R^2$. Then the system $T$ is not strongly Devaney chaos in weak topology. Besides, if $|\lambda|\neq 1$ and $|\lambda|\neq 0$, then the system $T$ is weakly Devaney chaos in weak topology. If else, the system $T$ is not weakly Devaney chaos in weak topology.
\end{Theorem}

\begin{proof}
By Theorem \ref{th3}, Theorem \ref{th6}, Theorem \ref{ths3} and  Theorem \ref{ths4}, we can get the result.

\end{proof}

\begin{Theorem}
Suppose that
$$
A=\left({\begin{array}{cc}
cos2\pi\theta&-sin2\pi\theta\\
sin2\pi\theta&cos2\pi\theta\\
\end{array}}\right)\left({\begin{array}{cc}
r&0\\
0&r\\
\end{array}}\right),
$$
where $r$ is a real number and $\theta\in[0,1]$. Consider the dynamical system $Tx=Ax$ with $x\in R^2$.

If $\theta$ is a rational number with $\theta\notin\{0,0.5,1\}$ and $r=1$, Then the system $T$ is strong Devaney chaos in weak topology. If else, the system $T$ is not weak Devaney chaos in weak topology.
\end{Theorem}

\begin{proof}
By  Theorem \ref{th4}, Theorem \ref{th7} and  Theorem \ref{ths5}, we can get the result.

\end{proof}

In class topological dynamical system, the famous Banks et al. theorem is not only simplifying the definition of Devaney chaos but also showing us the relationship between
transitivity, periodic points and sensitivity. The following theorem considers the Banks et al. theorem in two-dimensional space with weak topology.
\begin{Theorem}
In a two-dimensional space with weak topology, Banks et al. theorem still holds for strong Devaney chaos, However, Banks et al. theorem does not hold for weak Devaney chaos in general.
\end{Theorem}

\begin{proof}
By the results of  Section 3,  Section 4 and  Section 5, we can complete the proof.

\end{proof}

As we all know, in class topological dynamical system, if $T$ is transitive and has a fixed point, then it is Li-Yorke chaotic \cite{q79}. However, the following theorem shows that the result does not hold in two-dimensional space with weak topology.
\begin{Theorem}
There exists a  linear dynamical system $T:\mathbb{R}^2\to \mathbb{R}^2$ induced by matrix $A=\begin{pmatrix}a_{11}&a_{12}\\a_{21}&a_{22}\end{pmatrix}$ such that $T$ is weakly transitive in weak topology and has a fixed point in weak topology but not weakly Li–Yorke chaotic in weak topology.

\end{Theorem}

\begin{proof}
Let $A=\begin{pmatrix}1&0\\0&2\end{pmatrix}$. Then by Theorem \ref{th2} and Theorem \ref{th5}, $T$ is weakly transitive in weak topology and has a fixed point in weak topology. On the other hand, by \cite[Theorem 5]{q83}, $T$ not weakly Li–Yorke chaotic in weak topology.

\end{proof}


\end{document}